\documentclass[11pt,reqno,oneside]{amsart}

\usepackage{graphicx}
\usepackage{amssymb}
\usepackage{epstopdf}

\usepackage[a4paper, total={6in, 9.1in}]{geometry}

\usepackage{amsmath,amsfonts,amsthm,mathrsfs,amssymb,cite}
\usepackage[usenames]{color}

\usepackage{graphics}
\usepackage{framed}

\usepackage{enumerate}
\usepackage{hyperref}
\usepackage{xcolor}
\hypersetup{
  colorlinks,
  linkcolor={blue!80!black},
  urlcolor={blue!80!black},
  citecolor={blue!80!black}
}
\usepackage[capitalize,noabbrev]{cleveref}

\numberwithin{equation}{section}

\newtheorem{Theorem}{Theorem}[section]
\newtheorem{Lemma}[Theorem]{Lemma}

\numberwithin{equation}{section}

\def \Vh0{\stackrel{\circ}{V}_h}

\newcommand{\lc}
{\mathrel{\raise2pt\hbox{${\mathop<\limits_{\raise1pt\hbox
{\mbox{$\sim$}}}}$}}}

\newcommand{\gc}
{\mathrel{\raise2pt\hbox{${\mathop>\limits_{\raise1pt\hbox{\mbox{$\sim$}}}}$}}}

\newcommand{\ec}
{\mathrel{\raise2pt\hbox{${\mathop=\limits_{\raise1pt\hbox{\mbox{$\sim$}}}}$}}}

\def\bb{\begin{equation}} \def\ee{\end{equation}}

\def\beqn{\begin{eqnarray}}  \def\eqn{\end{eqnarray}}

\def\beqnx{\begin{eqnarray*}} \def\eqnx{\end{eqnarray*}}

\def\bn{\begin{enumerate}} \def\en{\end{enumerate}}

\def\bd{\begin{description}} \def\ed{\end{description}}



\title[Geometric structures of transmission eigenfunctions and applications]{On local and global structures of transmission eigenfunctions and beyond}

\author{Hongyu Liu}
\address{Department of Mathematics, City University of Hong Kong, Hong Kong SAR, China}
\email{hongyu.liuip@gmail.com, hongyliu@cityu.edu.hk}

\begin{document}

\maketitle

\centerline{\footnotesize\it Dedicated to Professor Michael Klibanov on the occasion of his 70th birthday }

\begin{abstract}

The (interior) transmission eigenvalue problems are a type of non-elliptic, non-selfadjoint and nonlinear spectral problems that arise in the theory of wave scattering. They connect to the direct and inverse scattering problems in many aspects in a delicate way. The properties of the transmission eigenvalues have been extensively and intensively studied over the years, whereas the intrinsic properties of the transmission eigenfunctions are much less studied. Recently, in a series of papers, several intriguing local and global geometric structures of the transmission eigenfunctions are discovered. Moreover, those longly unveiled geometric properties produce some interesting applications of both theoretical and practical importance to direct and inverse scattering problems. This paper reviews those developments in the literature by summarizing the results obtained so far and discussing the rationales behind them. There are some side results of this paper including the general formulations of several types of transmission eigenvalue problems, some interesting observations on the connection between the transmission eigenvalue problems and several challenging inverse scattering problems, and several conjectures on the spectral properties of transmission eigenvalues and eigenfunctions, with most of them are new to the literature.

\medskip



\end{abstract}

\section{Transmission eigenvalue problems}\label{sect:1}

We start with several general formulation of the (interior) transmission eigenvalue problems. Let $\lambda\in\mathbb{C}$ signify the eigenvalue which can be complex-valued. Introduce a general linear partial differential operator (PDO) $P(x, D)$ as follows:
\begin{equation}\label{eq:pdo1}
P(x, D)=\sum_{|\alpha|\leq m} a_\alpha(x) D^\alpha,\ \ x=(x_1, x_2,\ldots, x_n)\in\mathbb{R}^n,
\end{equation}
where the set of non-negative integers, $\alpha=(\alpha_1,\alpha_2,\ldots,\alpha_n)$, is called a multi-index, $|\alpha|=\alpha_1+\alpha_2+\ldots+\alpha_n$ called the length, $m\in\mathbb{N}$ called the order of the PDO, and 
\[
D^\alpha=D_1^{\alpha_1} D_2^{\alpha_2}\cdots D_n^{\alpha_n},  
\]
with
\[
D_j=\frac{1}{\mathrm{i}}\frac{\partial}{\partial x_j},\ j=1, 2, \ldots, n\quad \mbox{and}\quad \mathrm{i}:=\sqrt{-1}.
\]
In \eqref{eq:pdo1}, $a_\alpha:\Omega\rightarrow\mathbb{C}^{N\times N}$ is called the coefficient of $P(x, D)$, where $\Omega$ is a bounded domain in $\mathbb{R}^n$. We write $P_a(x, D)$ to signify its dependence on $a:=(a_\alpha)$. In order to appeal for a more general formulation, $a$ may also depend on $\lambda$ and in such a case, we write $P_a(x, \lambda, D)$ to reflect such a dependence. Let $\mathcal{H}(\Omega)$ be a function space from $\mathbb{R}^n$ to $\mathbb{C}$ and $\mathcal{H}^N$ be its $N$-copy. ${P}_a$ acts on $\mathcal{H}^N$ and it is assumed that for $u\in \mathcal{H}^N$, ${P}_a(x, \lambda, D) u$ is well-defined in the distributional sense. We further introduce the boundary trace operator ${T_a}: \mathcal{H}^N(\Omega)\rightarrow \mathcal{B}(\partial\Omega)$ associated with $P_a$, where $\mathcal{B}(\partial\Omega)$ denotes a certain function space on $\partial\Omega$. Let $T_a^j$, $j=1, 2, \ldots L$ be a set of trace operators and $\mathcal{B}_j$ be the associated trace function spaces. It is required that $T_a^j(0)=0$. 

Let $P_a$ and $P_b$ be two linear PDOs as introduced above and $a\equiv\hspace*{-3.5mm}\backslash\ b$. The general (interior) transmission eigenvalue problem is formulated as follows:
\begin{equation}\label{eq:tep1}
\begin{cases}
P_a(x, \lambda, D) u+\lambda u=0\ &\hspace*{-1cm} \mbox{in}\ \ \Omega,\medskip\\
P_b(x, \lambda, D) v+\lambda v=0\ & \hspace*{-1cm} \mbox{in}\ \ \Omega,\medskip\\
T_a^j(u)=T_b^j(v),\ \ j=1,2,\ldots, L,
\end{cases}
\end{equation}
where $u, v\in\mathcal{H}^N$. It is clear that $u=v\equiv 0$ are a pair of trivial solutions to \eqref{eq:tep1}. If there exist a pair of nontrivial solutions $(u, v)$ (namely, either $u\equiv\hspace*{-3.8mm}\backslash\ 0$ or $v\equiv\hspace*{-3.8mm}\backslash\ 0$), then $\lambda$ is called a transmission eigenvalue and $u, v$ are the associated transmission eigenfunctions. Throughout the rest of the paper, we set $k^2=\lambda$, namely $k=\sqrt{\lambda}$. We choose the complex branch such that $\Im k\geq 0$. 

We next present two specific transmission eigenvalue problems that have been widely studied in the literature. Let 
\begin{equation}\label{eq:helm1}
P_a=\Delta+k^2 V,\quad P_b=\Delta+k^2,
\end{equation}
where $V\in L^\infty(\Omega)$. In such a case, the transmission eigenvalue problem is
\begin{equation}\label{eq:th1}
\begin{cases}
\big(\Delta+k^2(1+V)\big) u=0\ &\ \mbox{in}\ \ \Omega,\medskip\\
(\Delta+k^2) v=0\ &\ \mbox{in}\ \ \Omega,\medskip\\
u=v,\ \ \partial_\nu u=\partial_\nu v\ &\ \mbox{on}\ \partial\Omega,
\end{cases}
\end{equation}
where $\nu\in\mathbb{S}^{n-1}$ signifies the exterior unit normal to $\partial\Omega$. The usual function space for the transmission eigenvalue problem \eqref{eq:th1} is the Sobolev space $H^1(\Omega)$ and in such a case, one may require that $\partial\Omega$ is Lipschitz. There is an alternative formulation of \eqref{eq:th1}. Set $w=u-v$. Consider the eigenvalue problem for $w\in H_0^2(\Omega)$:
\begin{equation}\label{eq:th2}
\big(\Delta+k^2(1+V)\big) (\Delta+k^2)w=0\quad\mbox{in}\ \ \Omega,
\end{equation}
which is equivalent to 
\begin{equation}\label{eq:th3}
\big(\lambda^2(1+V)+\lambda\cdot(2+V)\Delta+\Delta^2 \big) w=0\quad \mbox{in}\ \ H_0^2(\Omega),\quad \lambda=k^2. 
\end{equation}
It is easy to show that if $u, v\in H^2(\Omega)$, one deduces \eqref{eq:th3} from \eqref{eq:th2}. However, in general, one cannot deduce \eqref{eq:th2} from \eqref{eq:th3}. The reduced formulation \eqref{eq:th3} avoids the regularity requirement on $\partial\Omega$, but loses the track of the transmission eigenfunctions. It is (partially) evident from \eqref{eq:th3} that the transmission eigenvalue problem is non-elliptic, non-selfadjoint and nonlinear (in terms of the eigenvalue $\lambda$). This is in sharp difference to the traditional second order PDO eigenvalue problems, say e.g. the classical Dirichlet Laplacian eigenvalue problem
\begin{equation}\label{eq:dl1}
-\Delta u=\lambda u \quad \mbox{in}\ H_0^1(\Omega), 
\end{equation}
which is elliptic, selfadjoint and linear (in $\lambda$). It is pointed out that due to the non-selfadjointness of the transmission eigenvalue problem, there might exist complex transmission eigenvalues. Hence, even for the simplest transmission eigenvalue problem \eqref{eq:th1}, the corresponding study is interesting and challenging. The transmission eigenvalue problem \eqref{eq:th1} arises in the study of the direct and inverse acoustic scattering problems. We postpone the historical background introduction of transmission eigenvalue problems to the next section. 

Let  $\varepsilon=(\varepsilon_{ij})_{i,j=1}^3$ and $\mu= (\mu_{ij})_{i,j=1}^3$ be bounded and symmetric-positive-definite-matrix-valued functions in $\Omega$. Set
\begin{equation}\label{eq:max1}
P_{\varepsilon,\mu}=\begin{bmatrix}
k(\varepsilon-I) & \frac{1}{\mathrm{i}}\mathrm{curl}\\
-\frac{1}{\mathrm{i}}\mathrm{curl} & k(\mu-I)
\end{bmatrix},
\end{equation}
where for a $\mathbb{C}^3$-valued function $\mathbf{E}$, $\mathrm{curl}\,\mathbf{E}=\nabla\wedge\mathbf{E}$. Let $\mathbf{E}_1$ and $\mathbf{H}_1$ be $\mathbb{C}^3$-valued functions and $\mathbf{u}=[\mathbf{E}_1, \mathbf{H}_1]^T$ be a $\mathbf{C}^6$-valued function. Similarly, we let $\mathbf{v}=[\mathbf{E}_2, \mathbf{H}_2]^T$ be a $\mathbb{C}^6$-valued function. Let $(\varepsilon_j, \mu_j)$, $j=1, 2$, be two sets of parameters and consider the transmission eigenvalue problem
\begin{equation}\label{eq:max2}
\begin{cases}
P_{\varepsilon_1,\mu_1} \mathbf{u}+k \mathbf{u}=0& \mbox{in}\ \Omega,\medskip\\
P_{\varepsilon_2,\mu_2} \mathbf{v}+k \mathbf{v}=0 &\mbox{in}\ \Omega,\medskip\\
\nu\wedge\mathbf{E}_1=\nu\wedge\mathbf{E}_2,\ \nu\wedge\mathbf{H}_1=\nu\wedge\mathbf{H}_2 & \mbox{on}\ \partial\Omega. 
\end{cases}
\end{equation}
The function space for the transmission eigenvalue problem \eqref{eq:max2} is usually given by $H(\mathrm{curl}; \Omega)^2$, where
\[
H(\mathrm{curl}; \Omega):=\{\mathbf{E}\in L^2(\Omega)^3; \nabla\wedge\mathbf{E}\in L^2(\Omega)^3 \}. 
\]
\eqref{eq:max2} describes the transmission eigenvalue problem associated with the Maxwell system that arise in the study of the direct and inverse electromagnetic scattering problems. It is noted that the case $\varepsilon_2=\mu_2=I$ has been extensively studied in the literature.

We would like to single out a special case with the so-called strongly elliptic PDO, which is a second order PDO of the form
\begin{equation}\label{eq:ep1}
\mathcal{P}_A u:=-\sum_{j=1}^n\sum_{l=1}^n \partial_j(A_{jk}\partial_l u)+\sum_{j=1}^n A_j\partial_j u +Au\quad\mbox{in}\ \Omega,
\end{equation}
where the coefficients 
\[
A_{jl}=(a_{pq}^{jl})_{p, q=1}^N, \quad A_j=(a_{pq}^j)_{p, q=1}^N,\quad A=(a_{pq})_{p,q=1}^N. 
\]
Generically, $\mathcal{P}_A$ in \eqref{eq:ep1} can be reduced to $P_a$ in \eqref{eq:pdo1}, say e.g. by assuming the differentiability of the coefficient matrices $A_{jl}$. Nevertheless, due to its physical significance, we point it out separately. We refer to \cite{McL} for sufficient conditions on the coefficient matrices $A_{jl}$ such that $\mathcal{P}_A$ is strongly elliptic. Replacing $P_a$ and $P_b$ by $\mathcal{P}_{A}$ and $\mathcal{P}_B$ with $A\equiv\hspace*{-3.5mm}\backslash\, B$, one can have the corresponding transmission eigenvalue problem. In such a case, the boundary traces are usually given by the Cauchy data including the Dirichlet and Neumann data as follows:
\begin{equation}\label{eq:bt1}
T_A^1 (u)=u\quad \mbox{and}\quad T_A^2(u)=\sum_{j,l=1}^n \nu_j A_{jl}\partial_l u\quad \mbox{on}\ \ \partial\Omega,
\end{equation}
where $\nu=(\nu_j)_{j=1}^n$ is the exterior unit normal to $\partial\Omega$. We next present two specific examples within such a formulation. Let $g=(g_{ij})$ be a Riemannian metric on $\Omega$, namely a covariant symmetric 2-tensor, and $g^{-1}=(g^{ij})$. Introduce the Laplace-Beltrami operator, in local coordinates, as follows
\begin{equation}\label{eq:lb1}
\Delta_g u=\frac{1}{\sqrt{|g|}}\partial_i\big(\sqrt{|g|} g^{ij}\partial_j u \big). 
\end{equation}
Consider the transmission eigenvalue problem 
\begin{equation}\label{eq:thg1}
\begin{cases}
\big(\Delta_{g}+k^2(1+V)\big) u=0\ &\ \mbox{in}\ \ \Omega,\medskip\\
(\Delta_{\widetilde g}+k^2(1+\widetilde V)) v=0\ &\ \mbox{in}\ \ \Omega,\medskip\\
u=v,\ \ \partial_{\nu_g} u=\partial_{\nu_{\widetilde g}} v\ &\ \mbox{on}\ \partial\Omega,
\end{cases}
\end{equation}
where $(g, V)\equiv\hspace*{-4.1mm}\backslash\, (\widetilde g, \widetilde V)$. A particular case is given by taking $\widetilde g=g_0$ with $g_0$ signifying the Euclidean metric and $\widetilde V=0$, which is known as the transmission eigenvalue problem associated with the anisotropic Helmholtz equation.  The other example is the transmission eigenvalue problem associated with the Lam\'e operator. Let $\mathbf{C}=(C_{ijpq})_{i,j,p,q=1}^n$ be a 4-rank tensor satisfying 
\[
C_{ijpq}=C_{pqij}\quad\mbox{and}\quad C_{ijpq}=C_{jipq}=C_{ijqp},\quad i,j,p,q=1,2,\ldots, n. 
\]
$\mathbf{C}$ is known as the stiffness tensor of an elastic material. Introduce the Lam\'e operator as follows
\begin{equation}\label{eq:lame1}
\mathcal{L}_{\mathbf{C}} \mathbf{u}:=\nabla\cdot(\mathbf{C}:\nabla\mathbf{u})=\sum_{j,p,q=1}^n\partial_j\big(C_{ijpq}\partial_q u_p \big),\quad \mathbf{u}=(u_i)_{i=1}^n,
\end{equation}
Let $\rho=(\rho_{ij})_{i,j=1}^n$ be a symmetric-positive-definite-matrix-valued function, which signifies the density tensor of an elastic material. Let $(\mathbf{C}, \rho)\equiv\hspace*{-3.5mm}\backslash\, (\widetilde{\mathbf{C}}, \widetilde \rho)$ be two sets of elastic material tensors, and consider the following transmission eigenvalue problem 
\begin{equation}\label{eq:thc1}
\begin{cases}
\big(\mathcal{L}_{\mathbf{C}}+k^2\rho\big) \mathbf{u}=0\ &\ \mbox{in}\ \ \Omega,\medskip\\
(\mathcal{L}_{\widetilde{\mathbf{C}}}+k^2\widetilde\rho) \mathbf{v}=0\ &\ \mbox{in}\ \ \Omega,\medskip\\
\mathbf{u}=\mathbf{v},\ \ \partial_{\nu_{\mathbf{C}}} \mathbf{u}=\partial_{\nu_{\widetilde{\mathbf{C}}}} \mathbf{v}\ &\ \mbox{on}\ \partial\Omega,
\end{cases}
\end{equation}
where 
\[
\partial_{\nu_{\mathbf{C}}}\mathbf{u}:=\left( \sum_{j,p,q=1}^n \nu_jC_{ijpq}\partial_q u_p \right)_{i=1}^n
\]
is known as the traction of the elastic displacement field $\mathbf{u}$. It is noted that if $\rho=\alpha\cdot I$ and 
\[
C_{ijpq}=\beta \delta_{ij}\delta_{pq}+\gamma(\delta_{ip}\delta_{jq}+\delta_{iq}\delta_{jp}),
\]
with $\alpha, \beta$ and $\gamma$ being certain real scalar functions in $\Omega$, $(\mathbf{C}, \rho)$ is referred to as an isotropic elastic material. Here, $\delta$ is the conventional delta function. The transmission eigenvalue problem in the isotropic case can be formulated accordingly, which has received more studies in the literature. 

Finally, we present two more types of transmission eigenvalue problems, which to our knowledge, are largely new to the literature. The first one is called the partial-data transmission eigenvalue problems. We take \eqref{eq:th1} as a simple example to illustrate the formulation. Let $\Gamma, \Gamma'$ be two subsets of $\partial\Omega$. Consider the transmission eigenvalue problem
\begin{equation}\label{eq:th1p}
\begin{cases}
\big(\Delta+k^2(1+V)\big) u=0\ &\hspace*{-.3cm} \mbox{in}\ \ \Omega,\medskip\\
(\Delta+k^2) v=0\ &\hspace*{-.3cm} \mbox{in}\ \ \Omega,\medskip\\
u|_\Gamma=v|_{\Gamma},\ \ \partial_\nu u|_{\Gamma'}=\partial_\nu v|_{\Gamma'}. 
\end{cases}
\end{equation}
The formulation of \eqref{eq:th1p} is clearly more general than \eqref{eq:th1} and indeed, if $\Gamma=\Gamma'=\partial \Omega$, it reduces to \eqref{eq:th1}. We remark that it may happen that $\Gamma=\emptyset$ or $\Gamma'=\emptyset$. The formulation of the partial-data transmission eigenvalue problems can be easily extended to the case \eqref{eq:tep1}. Those new eigenvalue problems may have implications to the partial-data inverse boundary value problems which we shall discuss in what follows. The second one is to generalize  \eqref{eq:tep1} to the case where the PDOs $P_a$ and $P_b$ could be nonlinear. As an illustrating example, we present the following one:
\begin{equation}\label{eq:thp1}
\begin{cases}
\big(\Delta_{p}+k^2(1+V)\big) u=0\ &\ \mbox{in}\ \ \Omega,\medskip\\
\big(\Delta_{\widetilde p}+k^2(1+\widetilde V)\big) v=0\ &\ \mbox{in}\ \ \Omega,\medskip\\
u=v,\ \ \partial_{\nu_p} u=\partial_{\nu_{\widetilde p}} v\ &\ \mbox{on}\ \partial\Omega,
\end{cases}
\end{equation}
where the $p$-Laplacian for $1<p<\infty$ is defined by 
\[
\Delta_p u:=\nabla\cdot(|\nabla u|^{p-2}\nabla u)\quad\mbox{and}\quad \partial_{\nu_p} u=|\nabla u|^{p-2}\partial_\nu u. 
\]
It is assumed that $(p, V)\equiv\hspace*{-3.8mm}\backslash\ (\widetilde p, \widetilde V)$.

\section{Historical background and relevant discussions}

In the previous section, we introduced the general formulation of the transmission eigenvalue problems. In this section, we provide a general account of the historical background of the transmission eigenvalue problems as well as some relevant discussions. Throughout the rest of the paper, we shall mainly consider \eqref{eq:th1} for our discussion. Nevertheless, it is emphasized that the extension to the other types of transmission eigenvalue problems associated with different physical systems should be clear.

\subsection{Direct and inverse acoustic scattering problems}

We introduce the time-harmonic direct and inverse acoustic scattering problems (cf. \cite{CK}). Let $\eta\in L^\infty(\mathbb{R}^n)$, $n\geq 2$, be such that $\mathrm{supp}(\eta-1)\subset\Omega$. $\eta$ signifies the refractive index of an acoustic medium whose inhomogeneity is supported in $\Omega$. For simplicity, we assume that $\eta$ is real valued and for notational consistence, we set
\begin{equation}\label{eq:rf1}
\eta^2=1+V,\ \ \ \mbox{i.e.}\ \ V=\eta^2-1. 
\end{equation}
Clearly, $V\in L^\infty(\mathbb{R}^n)$ and $\mathrm{supp}(V)\subset\Omega$. Let $k\in\mathbb{R}_+$ denote the wavenumber of a time-harmonic acoustic wave and $u^i$ be an incident wave which is an entire solution to $(\Delta+k^2) u^i=0$ in $\mathbb{R}^n$. The impingement of the wave field $u^i$ on the acoustic scatterer $(\Omega, V)$ generates the acoustic scattering. Let $u$ and $u^s=u-u^i$ respectively denote the total and scattered wave fields. Then the acoustic scattering is described by the following Helmholtz system:
\begin{equation}\label{eq:as1}
\begin{cases}
& \Delta u+k^2(1+V) u=0\quad\mbox{in}\ \ \mathbb{R}^n,\medskip\\
& u=u^i+u^s,\medskip\\
&\displaystyle{\lim_{r\rightarrow\infty} r^{\frac{n-1}{2}}\left(\partial_r-\mathrm{i}k\right) u^s=0,}
\end{cases}
\end{equation}
where $r=|x|$, $x\in\mathbb{R}^n$, and $\partial_r u^s:=\hat x\cdot \nabla u^s(x)$, $\hat x:=x/|x|\in\mathbb{S}^{n-1}$. The last limit in \eqref{eq:as1} is known as the Sommerfeld radiation condition, which holds uniformly in the angular variable $\hat x$ and characterizes the outgoing nature of the scattered field. There exists a unique solution $u\in H_{loc}^2(\mathbb{R}^n)$ to \eqref{eq:as1} and moreover it admits the following asymptotic expansion (cf. \cite{CK}):
\begin{equation}\label{eq:far1}
u^s(x)=\frac{e^{\mathrm{i}k|x|}}{|x|^{(n-1)/2}} u_\infty(\hat x)+\mathcal{O}\left(\frac{1}{|x|^{(n+1)/2}}\right)\ \ \ \mbox{as}\ \ |x|\rightarrow\infty. 
\end{equation}
$u_\infty(\hat x)$ is known as the far-field pattern of $u^s$ and the correspondence between $u^s$ and $u_\infty$ is one-to-one. 

Let $u^i(x)=e^{\mathrm{i}kx\cdot d}$, $d\in\mathbb{S}^{n-1}$, which is known as a time-harmonic plane wave with $d$ signifying the incident direction. Set
\begin{equation}\label{eq:fn1}
u_\infty^k(\hat x, d)=u_\infty(\hat x; e^{\mathrm{i}kx\cdot d}),
\end{equation}
which denotes the far-field pattern associated with the plane wave. $u_\infty^k(\hat x, d)$, $(\hat x, d)\in \mathbb{S}^{n-1}\times\mathbb{S}^{n-1}$, is known as the scattering amplitude at $k\in\mathbb{R}_+$. Define the Herglotz operator $\mathcal{H}: L^2(\mathbb{S}^{n-1})\mapsto \mathscr{A}(\mathbb{R}^n)$ with $\mathscr{A}(\mathbb{R}^n)$ denoting the space of analytic functions over $\mathbb{R}^n$: 
\begin{equation}\label{eq:herg1}
\mathcal{H}(g)(x):=\int_{\mathbb{S}^{n-1}} e^{\mathrm{i}kx\cdot d} g(d)\, ds(d), \quad g\in L^2(\mathbb{S}^{n-1}). 
\end{equation}
$v_g(x)=\mathcal{H}(g)(x)$ is known as a Herglotz wave, which is the linear superposition of the plane waves with a density function $g(d)$. We have the following denseness property of the Hergoltz waves. 
\begin{Lemma}[\cite{Wec}]\label{lem:Herg}
Let $\Omega \Subset \mathbb R^n$ be a bounded Lipschitz domain and  ${\mathbf H}_k$ be the space of all the Herglotz wave functions of the form \eqref{eq:herg1}. Define 
$$
{\mathbf S}_k(\Omega ) =  \{u\in C^\infty (\Omega)~|~ \Delta u+k^2u=0\}
$$
and 
$$
{\mathbf H}_k(\Omega ) =  \{u|_\Omega~|~ u\in {\mathbf H}_k\}. 
$$
Then  ${\mathbf H}_k(\Omega )$ is dense in ${\mathbf S}_k(\Omega )  \cap  L^2 ( \Omega )$ with respect to the topology induced by the $H^1(\Omega)$-norm. 
\end{Lemma}
Using a Herglotz wave $v_g$ as an incident field, by the linearity of the Helmholtz system \eqref{eq:as1} with respect to the incident wave, one readily has
\begin{equation}\label{eq:herg2}
u_\infty(\hat x; v_g)=u_\infty(\hat x; \mathcal{H}(g(d)))=\mathcal{H}(u_\infty^k(\hat x, d))=\int_{\mathbb{S}^{n-1}} u_\infty^k(\hat x, d) g(d)\, ds(d). 
\end{equation}
Define the far-field operator $\mathcal{F}: L^2(\mathbb{S}^{n-1})\mapsto L^2(\mathbb{S}^{n-1})$ as follows:
\begin{equation}\label{eq:fo1}
\mathcal{F}(g)(\hat x)=\int_{\mathbb{S}^{n-1}} u_\infty^k(\hat x, d) g(d)\, ds(d). 
\end{equation}
By Lemma~\ref{lem:Herg}, we readily see that the far-field operator $\mathcal{F}$ actually contains all the possible scattering information from the scatterer $(\Omega, V)$. Hence, the direct scattering problem is to determine the far-field operator $\mathcal{F}$ for a given scatterer $(\Omega, V)$. Reversely, the inverse scattering problem is to determine the scatterer $(\Omega, V)$ by knowledge of the far-field operator:
\begin{equation}\label{eq:ip1}
\mathcal{F}\rightarrow (\Omega, V), 
\end{equation}
which is equivalent to determining $(\Omega, V)$ by knowledge of the associated scattering amplitude $u_\infty^k(\hat x, d)$. By introducing an operator $\mathcal{S}$ which sends to the scatterer $(\Omega, V)$ to the associated scattering amplitude $u_\infty^k(\hat x, d)$, the inverse scattering problem can be recast as the following operator equation
\begin{equation}\label{eq:ip2}
\mathcal{S}(\Omega, V)=u_\infty^k(\hat x, d). 
\end{equation}
It is directly verified that the inverse scattering problem \eqref{eq:ip1}/\eqref{eq:ip2} is nonlinear.

\subsection{Linear sampling method and transmission eigenvalue problem}\label{sect:2.2}

The linear sampling method (LSM) is a classical qualitative method for the inverse scattering problem \eqref{eq:ip2}, which makes use of the far-field data $u_\infty^k(\hat x, d)$ for all $\hat x, d\in\mathbb{S}^{n-1}$ but a fixed $k\in\mathbb{R}_+$ and aims to recover the support of scatterer, namely $\Omega$, independent of its content $V$. The LSM was first proposed in \cite{CKir} and has inspired a lot of subsequent studies \cite{CakCol, KirGrin}. The core of the LSM is the following far-field equation: 
\begin{equation}\label{eq:ff1}
\mathcal{F}(g)(\hat x)=\Phi_\infty(\hat x, z), \quad \Phi_\infty(\hat x, z):=\gamma_n e^{-\mathrm{i}k\hat x\cdot z}, \quad z\in\mathbb{R}^n,
\end{equation}
where $\Phi_\infty(\hat x, z)$ is the far-field pattern of $\Phi(x, z)$ and $\gamma_n$ is a dimensional constant. Here, $\Phi(x, z)$ is the outgoing fundamental solution to $-\Delta-k^2$ at $z\in\mathbb{R}^n$ given by
\[
\Phi(x, z)=\frac{1}{4\pi}\frac{e^{\mathrm{i}k|x-z|}}{|x-z|}\quad \mbox{when}\ \ n=3; \quad \frac{\mathrm{i}}{4} H_0^{(1)}(k|x-z|)\quad\mbox{when}\ \ n=2,
\]
where $H_0^{(1)}$ is the zeroth order Hankel function of the first kind. In general, the far-field equation \eqref{eq:ff1} is not solvable (exactly), especially considering that the integral kernel for $\mathcal{F}$, namely $u^k_\infty(\hat x, d)$, is given by measurement data with noise. Nevertheless, if $\mathcal{F}$ satisfies a certain ``generic condition", \eqref{eq:ff1} can be solved approximately, say by the Tikhonov regularization approach. Let $g_z^\varepsilon(\hat x)$ denote the approximate solution mentioned above, where $\varepsilon\ll 1$ signifies the regularization parameter and $z\in\mathbb{R}^n$ signifies a sampling point. The LSM uses $\|g_z^\varepsilon\|_{L^2(\mathbb{S}^{n-1})}$ as an indicator function for imaging $\Omega$, whose value is relatively large if $z\in\mathbb{R}^n\backslash\Omega$ and relatively small if $z\in\Omega$. In fact, theoretically, one has $\lim_{\varepsilon\rightarrow +0}\|g_z^\varepsilon\|_{L^2(\mathbb{S}^{n-1})}=\infty$ when $z\in \mathbb{R}^n\backslash \Omega$. Then the LSM works as follows. First, one selects a sampling mesh $\mathcal{T}_h$ containing the scatterer $\Omega$. Second, for each mesh grid point $z\in \mathcal{T}_h$, one solves the corresponding far-field equation \eqref{eq:ff1} to obtain $g_\varepsilon^z$. Finally, by selecting a cut-off value $c_0$, one can distinguish the interior and exterior of $\Omega$ according to the criterion: $z\in \Omega$ if $\|g_z^\varepsilon\|_{L^2(\mathbb{S}^{n-1})}\leq c_0$; and $z\in \mathbb{R}^n\backslash\Omega$ if $\|g_z^\varepsilon\|_{L^2(\mathbb{S}^{n-1})}> c_0$.

Before discussing more about the ``generic condition" on the far-field operator $\mathcal{F}$ ( which shall lead to the interior transmission eigenvalue problem), it is interesting to note that the inverse scattering problem \eqref{eq:ip2} is nonlinear, whereas solving \eqref{eq:ff1} is a linear process. Where is the nonlinearity of the inverse problem hidden in the LSM? In fact, the nonlinearity lies in deciding whether a given sampling point $z$ belongs to the interior or exterior of $\Omega$; that is, the selection of an appropriate cut-off value  is a nonlinear process, which becomes the most challenging part of the LSM. It is worth mentioning that a deterministic method of selecting an effective cut-off value was developed in \cite{LiLZ1} and an efficient multilevel procedure in classifying the interior and exterior sampling mesh points was developed in \cite{LLW1,LiLZ2}. Now, we return to the discussion on the aforementioned ``generic condition" which guarantees the ``approximate solvability" of the far-field equation \eqref{eq:ff1}. This is to require that the the far-field operator $\mathcal{F}$ has a dense range in $L^2(\mathbb{S}^{n-1})$. Since $\mathrm{Range}(\mathcal{F})=\mathrm{Ker}(\mathcal{F}^*)^\perp$, it is equivalent to requiring that $\mathcal{F}^*$ is injective. Here, 
\begin{equation}\label{eq:adj1}
\begin{split}
\mathcal{F}^*(g)(\hat x)=&\int_{\mathbb{S}^{n-1}} u_\infty^k(d, \hat x) g(d)\, ds(d)\\
=& \int_{\mathbb{S}^{n-1}} u_\infty^k(-\hat x, -d) g(d)\, ds(d), \quad \hat x\in\mathbb{S}^{n-1}, 
\end{split}
\end{equation}
where we make use of the reciprocity relation $u_\infty^k(\hat x, d)=u_\infty^k(-d, -\hat x)$ (cf. \cite{CK}). Suppose that there exists $g\in L^2(\mathbb{S}^{n-1})$ satisfying $\mathcal{F}^*(g)=0$, namely,
\begin{equation}\label{eq:adj2}
\int_{\mathbb{S}^{n-1}} u_\infty^k(\hat x, d) \tilde g(d)\, ds(d)=0,\quad \tilde g(d):=g(-d). 
\end{equation}
Then, by \eqref{eq:herg2}, one readily sees that $u_\infty(\hat x; v_{\tilde g})=0$. Let $u_{\tilde g}$ be the solution to the scattering system \eqref{eq:as1} with $u^i=v_{\tilde g}$ and $u_{\tilde g}^s=u_{\tilde g}-v_{\tilde g}$. By the Rellich theorem, we have from $u_\infty(\hat x; v_{\tilde g})=0$ that $u_{\tilde g}^s=0$ in $\mathbb{R}^n\backslash\overline{\Omega}$. By virtue of the transmission conditions of $u_{\tilde g}$ across $\partial\Omega$, one notes that
\begin{equation}\label{eq:transc1}
u_{\tilde g}\big|^-_{\partial\Omega}=u_{\tilde g}\big|^+_{\partial\Omega},\quad \partial_\nu u_{\tilde g} \big|^-_{\partial\Omega}=\partial_\nu u_{\tilde g}\big|^+_{\partial\Omega},
\end{equation}
where $\pm$ signifies the traces from the outside and inside of $\Omega$, respectively. Since $u_{\tilde g}^s=0$ in $\mathbb{R}^n\backslash\overline{\Omega}$, we have that
\begin{equation}\label{eq:transc2}
u_{\tilde g}\big|^+_{\partial\Omega}=v_{\tilde g}|^+_{\partial\Omega}=v_{\tilde g}|^-_{\partial\Omega},\quad \partial_\nu u_{\tilde g}\big|^+_{\partial\Omega}=\partial_\nu v_{\tilde g}|^+_{\partial\Omega}=\partial_\nu v_{\tilde g}|^-_{\partial\Omega}. 
\end{equation}
By combining \eqref{eq:transc1} and \eqref{eq:transc2}, we readily have that 
\begin{equation}\label{eq:thc1}
\begin{cases}
\big(\Delta+k^2(1+V)\big) u_{\tilde g}=0\ &\ \mbox{in}\ \ \Omega,\medskip\\
(\Delta+k^2) v_{\tilde g}=0\ &\ \mbox{in}\ \ \Omega,\medskip\\
u_{\tilde g}=v_{\tilde g},\ \ \partial_\nu u_{\tilde g}=\partial_\nu v_{\tilde g}\ &\ \mbox{on}\ \partial\Omega. 
\end{cases}
\end{equation}
That is, $(u_{\tilde g}|_{\Omega}, v_{\tilde g}|_{\Omega})$ is a pair of transmission eigenfunctions associated with the eigenvalue $k^2$. If $k^2$ is not a transmission eigenvalue associated with $(\Omega, V)$, then one has from \eqref{eq:thc1} that $v_{\tilde g}=0$ which in turn yields $\tilde g=0$ (cf. \cite{CK} ). Hence, for the good sake of the LSM, one should exclude the ``bad" transmission eigenvalues. In fact, this might be the major motivation for A. Kirsch to introduce and study the transmission eigenvalue problem in \cite{Kir}. One the other hand, we would like to point out that everything has two sides, and the ``failure" of the LSM at transmission eigenvalues can be used to the determine transmission eigenvalues by using the scattering data $u_\infty^k(\hat x, d)$ (cf. \cite{CCH,CHLW1}). 

\subsection{Spectral properties of the transmission eigenvalues }

The spectral properties of transmission eigenvalues have been extensively and intensively studied in the literature, and in many aspects, they resemble those for the classical Dirichlet/Neumann Laplacian. For example, the transmission eigenvalues are discrete, infinite and accumulating only at $\infty$, and for each transmission eigenvalue, the corresponding transmission eigenspace is finite dimensional. On the other hand, due to the distinct features of the transmission eigenvalue problems, there are some particular and unique properties possessed by the transmission eigenvalues, say e.g. the existence, discreteness, infiniteness of the transmission eigenvalues as well as the corresponding counting formula not only depend on the geometry of $\Omega$ but also critically rely on the inhomogeneous parameter $V$; and moreover, there exist complex eigenvalues due to the non-self-adjointness. This is a vibrant field of research with abundant results in the literature. There are papers surveying and reviewing those developments and we refer to \cite{CCH,CH,CKreview} and the references cited therein for the existing results in the literature and related open problems on the transmission eigenvalues.

\section{More connections to inverse scattering problems and invisibility cloaking}

In this section, we present more observation and discussion from our perspective on the connections between the transmission eigenvalue problems and several challenging problems related to inverse scattering theory and invisibility cloaking.

\subsection{Unique identifiability for inverse scattering problems}

Let us consider again the inverse scattering problem \eqref{eq:ip1}/\eqref{eq:ip2}. We count the cardinalities of the unknown scatterer $(\Omega, V)$ and the measurement data, respectively. Here, by cardinality, we mean the number of independent variables in a quantity. It is clear that the cardinality of $(\Omega, V)$ is $n$ in the generic case, and  $\{u_\infty^k(\hat x, d)\}_{k\in\mathbb{R}_+, d\in\mathbb{S}^{n-1}, \hat x\in\mathbb{S}^{n-1}}$ is $2n-1$. Hence, the inverse problem is overdetermined when full measurement data are used (in such a case, $2n-1>n$ for $n\geq 2$). In order to establish the unique identifiability of the inverse problem \eqref{eq:ip1}/\eqref{eq:ip2}, it would be helpful to introduce the following inverse boundary value problem. To that end, we introduce the following Cauchy data set:
\begin{equation}\label{eq:cauchy}
\mathcal{C}_V:=\{u|_{\partial\Omega}, \partial_\nu u|_{\partial\Omega}\}\in H^{1/2}(\partial\Omega)\times H^{-1/2}(\partial\Omega), 
\end{equation}
where $u\in H^1(\Omega)$ is any solution to 
\begin{equation}\label{eq:ee1}
\big(\Delta+k^2(1+V)\big)u=0\ \mbox{in}\ \Omega.
\end{equation} 
It is known that for any fixed $k\in\mathbb{R}_+$ (cf. \cite{CK}), 
\begin{equation}\label{eq:cauchy2}
\mathcal{F}\Longleftrightarrow \mathcal{C}_V. 
\end{equation}
In what follows, we shall also write $\mathcal{F}_V^k$ or $\mathcal{C}_V^k$ to specify the dependence on the wavenumber $k$ as well as the medium parameter $V$. 

\begin{Theorem}\label{thm:1}
Consider the inverse scattering problem \eqref{eq:ip1}/\eqref{eq:ip2} for $n\geq 3$. $\mathcal{F}_V^k$ for all $k\in\mathbb{R}_+$ uniquely determines $V$. 
\end{Theorem}

\begin{proof}
By \eqref{eq:cauchy2}, it is sufficient for us to prove that $\mathcal{C}_V^k$ uniquely determines $V$. Let $V_j\in L^\infty(\Omega)$, $j=1, 2$, be such that
\[
\mathcal{C}_{V_1}^k=\mathcal{C}_{V_2}^k\quad\forall k\in\mathbb{R}_+. 
\]
Then one can show that the following integral identity (cf. \cite{SU})
\begin{equation}\label{eq:identity1}
\int_{\Omega} (V_1-V_2) u_1 u_2=0,
\end{equation}
which holds for any $u_j\in H^1(\Omega)$ satisfying 
\begin{equation}\label{eq:identity2}
\big(\Delta+k^2 (1+V_j)\big) u_j=0,\quad j=1,2.
\end{equation}
Note that \eqref{eq:identity1} and \eqref{eq:identity2} hold for any $k\in\mathbb{R}_+$. By letting $k\rightarrow+0$, one can show that $u_j\rightarrow u_j^0$ where $\Delta u_j^0=0$ in $\Omega$ (cf. \cite{LU}). Hence, we have from \eqref{eq:identity1} by passing $k\rightarrow +0$ that
\begin{equation}\label{eq:identity3}
\int_\Omega (V_1-V_2) u_1^0 u_2^0=0,  
\end{equation} 
where $u_j^0$, $j=1, 2$, are two arbitrary harmonic functions. By Calder\'on's classical argument \cite{C,SU}, one can construct solutions of the form:
\begin{equation}\label{eq:identity4}
u^0_j(x)=\exp(\rho_j\cdot x),\ \ \rho_j\in\mathbb{C}^n, \ \rho_j\cdot \rho_j=0,\ j=1, 2;\ \ \rho_1+\rho_2=\mathrm{i}\xi,\ \forall\xi\in\mathbb{R}^n. 
\end{equation}
Substituting \eqref{eq:identity4} into \eqref{eq:identity3}, one has
\begin{equation}\label{eq:identity5}
\int_\Omega (V_1-V_2)(x) \exp(\mathrm{i}x\cdot \xi)\, dx=0, 
\end{equation}
which readily yields that $V_1=V_2$. 
\end{proof}

It is remarked that the uniqueness result in Theorem~\ref{thm:1} also holds for $n=2$ since the products of harmonic functions are also dense in $L^2(\mathbb{R}^2)$. However, in order to prove this, one needs to make use of different techniques other than the construction in \eqref{eq:identity4} (cf. \cite{AP,IUY1}). It is interesting to point out that to our best knowledge, the unique identifiability result in Theorem~\ref{thm:1} and especially its proof, although simple, were not given in the literature. The uniqueness result in Theorem~\ref{thm:1} readily inspires a unique identifiability result associated with the time-dependent wave equation: 
\begin{equation}\label{eq:td1}
\frac{1}{c^2}\partial_t^2 w-\Delta w=0\quad \mbox{in}\ \ \Omega,
\end{equation}
where $c\in L^\infty(\Omega)$ and $|c|>c_0\in\mathbb{R}_+$.  
The measurement data for the inverse problem associated with \eqref{eq:td1} is the following dynamical Cauchy data set:
\begin{equation}\label{eq:td2}
\Lambda_c^t:=\{w|_{\partial\Omega}, \partial_\nu w|_{\partial\Omega}\},\ t\in\mathbb{R}_+,
\end{equation}
where $w(x, t)\in H^2(\mathbb{R}_+, H^1(\Omega))$ is any solution to \eqref{eq:td1}. The inverse problem is to recover $c$ by knowledge of $\Lambda_{c}^t$ for all $t\in\mathbb{R}_+$. Assuming the following temporal Fourier transform exists for $w(x, t)$ (by imposing a certain condition on $(\Omega, c)$):
\[
u(x, k)=\frac{1}{2\pi}\int_0^\infty w(x, t) e^{ik t}\ dt,\quad (x, k)\in\mathbb{R}^n\times \mathbb{R}_+,
\]
then \eqref{eq:td2} is reduced to \eqref{eq:ee1} with $1+V=c^{-2}$. Consequently, $(\Lambda_c^t, t\in\mathbb{R}_+)$ is equivalently reduced to $(\mathcal{C}_V^k, k\in\mathbb{R}_+)$. Hence, Theorem~\ref{thm:1} implies that $(\Lambda_c^t, t\in\mathbb{R}_+)$ uniquely determines $c$. It is clear that the dataset $(\Lambda_c^t, t\in\mathbb{R}_+)$ is overdetermined in determining $c$ (since the cardinality of $(\Lambda_c^t, t\in\mathbb{R}_+)$ is $2n-1$, and for a generic $c$, its cardinality is $n$). We would like to point out the connection to an important class of inverse problems in the literature, namely the so-called inverse problem with a single measurement. It asks whether one can uniquely determine $c$ by knowledge of $(w(x, t), \partial_\nu w(x, t))|_{(x, t)\in\partial\Omega\times\mathbb{R}_+}$ for a fixed solution $w(x, t)$ to \eqref{eq:td1}. The single-measurement dataset is obviously formally-determined in recovering $c$. The single-measurement inverse problems were resolved in generic scenarios by the classical Bukhgeim-Klibanov method \cite{BuhKli} of using tools from Carleman estimates; see also \cite{K, KT} and the references cited therein for many subsequent developments.   

Similar to our discussion above, for the inverse problem \eqref{eq:ip1}/\eqref{eq:ip2}, the dataset $(\mathcal{F}_V^k, k\in\mathbb{R}_+)$ is overdetermined for the determination of $V$ (recalled that $2n-1>n$, $n\geq 2$). It is natural to ask what is the optimal/minimal dataset in determining $V$. If one uses $\mathcal{F}_V^k$ for a fixed $k\in\mathbb{R}_+$ (which is equivalent to $\mathcal{C}_V^k$ for a fixed $k\in\mathbb{R}_+$), by direct counting, the cardinality for the measurement dataset is $2n-2$, whereas the cardinality of the unknown $V$ is $n$. Hence, the inverse problem is still overdetermined when $n\geq 3$ and it is formally-determined if $n=2$. In such a case, the uniqueness results were established in the seminal works \cite{SU} for $n\geq 3$ and \cite{AP} for $n=2$. However, the measurement dataset is still overdetermined (except for $n=2$). Next, we consider the following dataset
\begin{equation}\label{eq:fd1}
\Theta_V:=\{u^k_\infty(\hat x, d);\ \mbox{a fixed $d\in\mathbb{S}^{n-1}$ and all $(k, \hat x)\in\mathbb{R}_+\times\mathbb{S}^{n-1}$} \}. 
 \end{equation}
 The cardinality of $\Theta_V$ is $n$ and hence the inverse problem of determining $V$ by knowledge of $\Theta_V$ is formally determined. Physically, $\Theta_V$ is obtained by sending incident plane waves at a fixed impinging direction with varying frequencies and collecting the far-field data at all observation directions. By the reciprocity relation, namely $u_\infty^k(\hat x, d)=u_\infty^k(-\hat d, -\hat x)$, it can also be obtained by sending incident plane waves at all possible impinging directions with varying frequencies and collecting the far-field data at a fixed spot. For the unique identifiability issue, we let $V_j\in L^\infty(\Omega)$, $j=1, 2$, be such that 
 \begin{equation}\label{eq:fd2}
 \Theta_{V_1}=\Theta_{V_2}. 
 \end{equation}
Let $u_j$ be the total wave field of \eqref{eq:as1} associated with $V_j$ and $u^i=e^{\mathrm{i}kx\cdot d}$ for a fixed $d\in\mathbb{S}^{n-1}$ and all $k\in\mathbb{R}_+$. By \eqref{eq:fd2}, one can directly show that there holds the following transmission eigenvalue problem:
\begin{equation}\label{eq:thu1}
\begin{cases}
\big(\Delta+k^2(1+V_1)\big) u_1=0\ &\ \mbox{in}\ \ \Omega,\medskip\\
\big(\Delta+k^2(1+V_2)\big) u_2=0\ &\ \mbox{in}\ \ \Omega,\medskip\\
u_1=u_2,\ \ \partial_{\nu} u_1=\partial_{\nu} u_2\ &\ \mbox{on}\ \partial\Omega. 
\end{cases}
\end{equation}
Clearly, \eqref{eq:thu1} holds for all $k\in\mathbb{R}_+$. According to our earlier discussion, the transmission eigenvalues should be discrete. If this is the case, one can readily have a contradiction from \eqref{eq:thu1} unless $V_1=V_2$. That is, one can derive the unique identifiability result for determining $V$ by knowledge of $\Theta_V$. However, there is a technical challenge that the discreteness of the transmission eigenvalues for \eqref{eq:thu1} critically depends on the a-priori forms of $V_1$ and $V_2$. In fact, for all of the existing results on this topic, it is required that $V_1$ and $V_2$ are different from each other near $\partial\Omega$ (cf. \cite{CGH,CPS,LV,PS,S,R}). But for the transmission eigenvalue problem \eqref{eq:thu1} arising from the inverse problem, it would be too restrictive of imposing such a condition on $V_1$ and $V_2$. Nevertheless, we have the following conjecture:

\smallskip

\noindent{\bf Conjecture 1.}~\emph{Consider the inverse scattering problem \eqref{eq:ip2}. $\Theta_V$ unique determines $V$ generically. }

\smallskip

We would like to remark that the inverse problem in Conjecture~1 as well as its connection to the transmission eigenvalue problem \eqref{eq:thu1} were first proposed and investigated in \cite{HLL}. It is shown in two special cases: (1)~$V$ is constant;~(2)~$V$ is spherically symmetric that $V$ can be uniquely determined by $\Theta_V$. We believe that Conjecture~1 should hold for more general cases. To cast some light on this conjecture, one may consider the case that $V$ is a real-analytic function in $\Omega$. In order to establish the unique identifiability result in such a case, by a similar argument, one has the transmission eigenvalue problem \eqref{eq:thu1}. Since $V_1$ and $V_2$ are analytic in $\Omega$, there must be derivatives of a certain finite order for $V_1$ and $V_2$ on $\partial\Omega$ such that they are different from each, since otherwise all the derivatives of $V_1$ and $V_2$ on $\partial\Omega$ are the same, one immediately has $V_1=V_2$ in $\Omega$ by the analytic continuation. Then, under such a condition, it is very plausible to show the discreteness of the transmission eigenvalues for \eqref{eq:thu1}, which gives rise to a contradiction and in turn implies that $V_1=V_2$ in $\Omega$. Moreover, there is another perspective that can increase the plausibility of Conjecture~1. In fact, from our earlier discussion about the connection between the single-measurement inverse problem associated with the time-dependent wave equation \eqref{eq:td1} and its time-harmonic counterpart associated with \eqref{eq:ee1}, we readily see that $\Theta_V$ is actually equivalent to the single-measurement data of its time-dependent counterpart. Those are intriguing topics that are worth further investigations. 

Finally, we propose a partial-data inverse boundary value problem that is related to the partial-data transmission eigenvalue problem \eqref{eq:th1p}. Let $\Gamma$ and $\Gamma'$ be two open subsets of $\partial\Omega$. Associated with the Helmhotlz equation \eqref{eq:ee1}, we introduce the partial-data Cauchy set:
\begin{equation}\label{eq:cauchyp}
\mathcal{C}_V^{\Gamma, \Gamma'}:=\{u|_{\Gamma}, \partial_\nu u|_{\Gamma'}\}\in H^{1/2}(\Gamma)\times H^{-1/2}(\Gamma'), 
\end{equation}
where $u\in H^1(\Omega)$ is a solution to \eqref{eq:ee1}. The inverse boundary problem is to determine $V$ by knowledge of $\mathcal{C}_{V}^{\Gamma, \Gamma'}$. If $\Gamma=\Gamma'=\partial\Omega$, $\mathcal{C}_{V}^{\Gamma, \Gamma'}$ is reduced to $\mathcal{C}_V$ in \eqref{eq:cauchy}. There are results showing that if $\mathcal{C}_{V}^{\Gamma, \Gamma'}$ is given for all solutions to \eqref{eq:ee1} and a fixed $k\in\mathbb{R}_+$, under certain scenarios where $\Gamma, \Gamma'$ could be proper subsets of $\partial\Omega$, $V$ can be uniquely recovered; see \cite{IUY1,KSU}. However, it can be easily shown that in such a case the data used are over-determined when $n\geq 3$. Next, we consider a formally determined case. Set 
\begin{equation}\label{eq:cauchyp2}
\begin{split}
\mathcal{D}_V^k(\psi):=&\{ (u|_{\Gamma}, \partial_\nu u|_{\Gamma'}); u\in H^1(\Omega)\ \mbox{satisfies \eqref{eq:ee1}}\\
&\hspace*{2.5cm} \mbox{with $u|_{\partial \Omega}=\psi\in H^{1/2}(\partial\Omega)$ and $\mathrm{supp}(\psi)\subset\Gamma$} \}. 
\end{split}
\end{equation}
Let consider the following inverse problem
\begin{equation}\label{eq:pi1}
\mathcal{D}_V^k(\psi)\ \mbox{with a fixed $\psi$ and all $k\in\mathbb{R}_+$}\rightarrow V. 
\end{equation}
It can be shown that the inverse problem \eqref{eq:pi1} is formally determined. To our knowledge, this type of inverse problem is new to the literature. In order to establish the corresponding unique identifiability result. Suppose $V_1, V_2\in L^\infty(\Omega)$ fulfil that
\begin{equation}\label{eq:pi2}
\mathcal{D}_{V_1}^k(\psi)=\mathcal{D}_{V_2}^k(\psi)\ \ \mbox{for a fixed $\psi$ and all $k\in\mathbb{R}_+$. }
\end{equation}
One readily has from \eqref{eq:pi2} the following partial-data transmission eigenvalue problem:
\begin{equation}\label{eq:thu1pp1}
\begin{cases}
\big(\Delta+k^2(1+V_1)\big) u_1=0\ &\hspace*{-2cm} \mbox{in}\ \ \Omega,\medskip\\
\big(\Delta+k^2(1+V_2)\big) u_2=0\ &\hspace*{-2cm} \mbox{in}\ \ \Omega,\medskip\\
u_1|_{\Gamma}=u_2|_{\Gamma}=\psi|_{\Gamma},\ \ \partial_{\nu} u_1|_{\Gamma'}=\partial_{\nu} u_2|_{\Gamma'}, 
\end{cases}
\end{equation}
which holds for all $k\in\mathbb{R}_+$. Hence, if one can show the discreteness of the transmission eigenvalues for \eqref{eq:thu1pp1}, one can establish the unique identifiability for \eqref{eq:pi2}, namely $V_1=V_2$.

\subsection{Invisibility in wave scattering}\label{sect:3.2}

In this subsection, we discuss the connection between the invisibility in wave scattering and the transmission eigenvalue problems. Invisibility cloaking is concerned with certain mechanisms of making a target object undetectable with respect to certain wave probing means. This is a huge topic and we shall only consider a particular aspect, namely its connection to the transmission eigenvalue problems. 

We start our discussion by considering the scattering problem \eqref{eq:as1}. Let $u_\infty(\hat x; e^{\mathrm{i}kx\cdot d})$ be the far-field pattern defined in \eqref{eq:far1}  corresponding to the total field $u$ and incident field $u^i=\exp(\mathrm{i}kx\cdot d)$. If it happens that 
\begin{equation}\label{eq:i1}
u_\infty(\hat x; u^i)=0,\quad \forall x\in\mathbb{S}^{n-1}, 
\end{equation} 
we say that invisibility occurs. Physically, it means that the impingement of the incident field $u^i$ generates no scattering information of the scatterer $(\Omega, V)$ in the far-field observation. In such a case, one has from \eqref{eq:i1} by Rellich's theorem that $u^s(x)=0$ for $x\in\mathbb{R}^n\backslash\overline{\Omega}$. Hence, by a similar argument in deriving \eqref{eq:thc1}, one can show that:
\begin{equation}\label{eq:thcn1}
\begin{cases}
\big(\Delta+k^2(1+V)\big) u=0\ &\ \mbox{in}\ \ \Omega,\medskip\\
(\Delta+k^2) u^i=0\ &\ \mbox{in}\ \ \Omega,\medskip\\
u=u^i,\ \ \partial_\nu u=\partial_\nu u^i\ &\ \mbox{on}\ \partial\Omega. 
\end{cases}
\end{equation}
The same result holds for any $u^i$ which is an entire solution to $(\Delta+k^2) u^i=0$ in $\mathbb{R}^n$, not necessarily being the plane wave. Moreover, it can be shown (cf. \cite{BL1,CHLW1}) 
that invisibility occurs for an incident field $u^i$ if and only if $(u|_{\Omega}, u^i|_{\Omega})$ is a pair of transmission eigenfunctions associated with the transmission eigenvalue $k^2$ and the target scatterer $(\Omega, V)$. Hence, for the good sake of invisibility, one should expect the transmission eigenvalues and transmission eigenfunctions of \eqref{eq:thcn1} should be as ``dense" as possible, and this is in sharp difference to that for the LSM which requires the ``sparse" the better; see our discussion in Section~\ref{sect:2.2}. 

There are many delicate but challenging issues regarding the above connection, which are fundamental to invisibility cloaking. We only mention two of them in what follows. First, we let  
$(u|_{\Omega}, u^i|_{\Omega})$ be a pair of transmission eigenfunctions of \eqref{eq:thcn1} associated with $k^2$. In order for perfect invisibility occurs, namely \eqref{eq:i1} holds, the transmission eigenfunction $u^i|_{\Omega}$ should be (analytically) extended to an entire solution satisfying $(\Delta+k^2)u^i=0$ in $\mathbb{R}^n$. However, it is widely believed that this cannot be true in general. Indeed, we have the following conjecture:
\smallskip

\noindent{\bf Conjecture 2.}~\emph{Consider the scattering problem \eqref{eq:as1} with $V\in L^\infty(\Omega)$ and $u^i$ being a nontrivial entire solution to $(\Delta+k^2)u^i=0$ in $\mathbb{R}^n$. Then $u_\infty(\hat x; u^i)\equiv\hspace*{-3mm}\backslash\, 0$ in general.  }

\smallskip

According to our discussion above, Conjecture 2 basically indicates that an inhomogeneous medium $(\Omega, V)$ generically scatters every incident wave nontrivially; namely, perfect invisibility cannot be achieved. It is remarked that there are simple counter examples to Conjecture 2, say e.g. $\Omega$ is a ball and $V$ is spherically symmetric. Nevertheless, Conjecture 2 should hold generically excluding those simple cases. In fact, this fundamental problem has received considerable attentions in the literature recently since the seminal work \cite{BPS}. It is shown in \cite{BPS} that if $\partial\Omega$ has a corner, then the transmission eigenfunction $u^i|_{\Omega}$ cannot be analytically extended across the corner. There are many nice developments following a similar spirit in the literature \cite{B,BLin,BL2,BL4,BXL,CX,LX2} and most of them require certain ``singularities" of $\partial\Omega$ and $V$ near the boundary, except in \cite{BL2} where $\partial\Omega$ can be very smooth or even analytic, but possesses a high-curvature point. It is also remarked that all of the aforementioned results are qualitative except \cite{BL4} where a quantitative stability estimate was established.  On the other hand, we would like to mention in passing that there exist ``abnormal" materials such that Conjecture 2 holds true. Here by ``abnormal", we mean the material parameters (namely $V$ in the current case) are no longer bounded and even anisotropic. We refer to \cite{GKLU1,GKLU2,HL1,L4,LU2} for some specific examples of such invisible ``abnormal" materials as well as surveys on the existing developments on this topic. Hence, Conjecture~2 basically implies that there is no hope in creating a cloaking device by ``normal" materials, namely those materials exist in the nature. 

The other issue is related to the partial-data transmission eigenvalue problem (cf. \eqref{eq:th1p} and \eqref{eq:thu1pp1}), which naturally arises in the study of invisibility cloaking. To illustrate the situation, we 
consider the invisibility cloaking with respect to boundary measurements. Let $(\Omega, V)$ denote the inhomogeneous medium in $\Omega$ and consider the following transmission eigenvalue problem
\begin{equation}\label{eq:th1pn}
\begin{cases}
\big(\Delta+k^2(1+V)\big) u=0\ &\hspace*{-.3cm} \mbox{in}\ \ \Omega,\medskip\\
(\Delta+k^2) u^i=0\ &\hspace*{-.3cm} \mbox{in}\ \ \Omega,\medskip\\
u|_\Gamma=u^i|_{\Gamma},\ \ \partial_\nu u|_{\Gamma'}=\partial_\nu u^i|_{\Gamma'}, 
\end{cases}
\end{equation}
where $\Gamma$ and $\Gamma'$ are two subsets of $\partial\Omega$. Physically,\eqref{eq:th1pn} can be interpreted as follow. By inputting the probing wave $u^i$ on $\Gamma$, one measures the output wave on $\Gamma'$. Due to the transmission conditions in \eqref{eq:th1pn}, one can see that invisibility cloaking effect is achieved with respect to such partial-boundary measurements. Let $\mathscr{T}_{\Gamma, \Gamma'}$ denote the set of transmission eigenvalues to \eqref{eq:th1pn}, and if $\Gamma=\Gamma'$, $\mathscr{T}_{\Gamma}:=\mathscr{T}_{\Gamma, \Gamma}$. It is clear that $\mathscr{T}_{\partial\Omega}$ is the set of the full-data transmission eigenvalues and $\mathscr{T}_{\emptyset}=\mathbb{R}_+$ (here, we only consider the real and nonzero transmission eigenvalues). It is conjectured that:
\smallskip

\noindent{\bf Conjecture 3.}~\emph{Consider the partial-data transmission eigenvalue problem \eqref{eq:th1pn}. Let $\Gamma_j\subset\partial\Omega$ be open subsets and $\Gamma_1\subsetneq\Gamma_2$. Generically, one has $\mathscr{T}_{\Gamma_1}\supsetneq\mathscr{T}_{\Gamma_2}$. Here, the generic conditions should be imposed on $\Gamma_1, \Gamma_2$ and $V$. }

\smallskip

In a similar manner, Conjecture 3 can be formulated for $\mathscr{T}_{\Gamma, \Gamma'}$ with $\Gamma\neq \Gamma'$. Finally, we mention in passing some other related studies and interesting observations on the connections between the invisibility cloaking and transmission eigenvalue problems \cite{JL,LLGW,L1,LWZ}. 

\section{Geometric structures of transmission eigenfunctions}

In this section, we discuss the geometric properties of transmission eigenfunctions. The geometric structures of eigenfunctions have been a central topic in different branches of the spectral theory. To comprehend this point, even for the ``simplest" Dirichlet/Neumann Laplacian eigenfunctions, there are many celebrated conjectures with many ongoing studies including the Courant's nodal domain theorem, the hot-spot conjecture and the Schiffer conjecture, to name just a few; see \cite{CDLZ1, CDLZ2} for a related background discussion on those conjectures. We also refer to \cite{CDLZ2,CDLZ1} for brief reviews as well as some new developments on this important topic. The geometric structures of transmission eigenfunctions were less touched in the literature until recently the author of the present paper and coauthors initiate the study and discover several delicate and intriguing local and global structures of transmission eigenfunctions. In what follows, we discuss those developments as well as some relevant results by other researchers.

\subsection{Local structures of the transmission eigenfunctions}

\cite{BL1} presents the first discovery on the intrinsic geometric structure of transmission eigenfunction. It is proved that the transmission eigenfunctions generically vanish near a corner point, and the result is summarized in the following theorem. 

\begin{Theorem}[\cite{BL1}]\label{thm:local1}
Let $x_c\in\partial\Omega$ be a corner point and let $\varphi\in C^\alpha(\mathbb{R}^n)$ for some $\alpha>0$ (which can be complex-valued) satisfying $\varphi(x_c)\neq 0$. Suppose that there exists a neighbourhood of $x_c$, $B_{\epsilon}(x_c)$, such that $V=\phi$ in $B_\epsilon(x_c)\cap \Omega$. Consider the transmission eigenvalue problem \eqref{eq:th1} with $u, v\in H^1(\Omega)$. Then under a ``generic condition", it holds that
\begin{equation}\label{eq:vanish1}
\lim_{r\rightarrow+0}\frac{1}{|B_r(x_c)\cap \Omega|}\int_{B_r(x_c)\cap \Omega} |v(x)|\, dx=0.
\end{equation}
\end{Theorem}

We postpone the description of the ``generic condition" in Theorem~\ref{thm:local1} and firstly sketch the main idea in proving the theorem. To that end, we present the following lemma which connects the transmission eigenvalue problem to the scattering problem \eqref{eq:as1}. 

\begin{Lemma}[\cite{BL1}]\label{lem:local1}
Consider the transmission eigenvalue problem \eqref{eq:th1} with $u, v\in H^1(\Omega)$. According to Lemma~\ref{lem:Herg}, we let $v_g$ be a Herglotz wave function such that 
\begin{equation}\label{eq:app1}
\|v_g-v\|_{L^2(\Omega)}\leq \varepsilon\ll 1. 
\end{equation}
Consider the scattering problem \eqref{eq:as1} with $u^i=v_g$. Then it holds that
\begin{equation}\label{eq:app2}
\|u_\infty\|_{L^2(\mathbb{S}^{n-1})}\leq C \varepsilon,
\end{equation}
where $C$ depends on $\Omega, V$ and $k$. 
\end{Lemma}

Now, the proof of Theorem~\ref{thm:local1} can be divided into the following four steps:\smallskip

{\bf Step 1.}~Consider the scattering problem \eqref{eq:as1} associated with the scatterer $(\Omega, V)$ described in Theorem~\ref{thm:local1} and a generic incident wave $u^i$. It is shown in \cite{BPS} that 
\begin{equation}\label{eq:est1}
u_\infty(\hat x; u^i)\equiv\hspace*{-3mm} \backslash\, 0. 
\end{equation}
According to our discussion after Conjecture 2, \eqref{eq:est1} indicates that if a scattering medium $(\Omega, V)$ possesses a corner, it scatters every incident wave nontrivially. 

\smallskip

{\bf Step 2.}~In \cite{BL4}, the qualitative result \eqref{eq:est1} is quantified. It is shown that there exists a stability function $\psi$ (which is real valued and satisfies $\lim_{t\rightarrow+0}\psi(t)=0$) such that
\begin{equation}\label{eq:ss1}
\|u_\infty(\hat x; u^i)\|_{L^2(\mathbb{S}^{n-1})}\geq \psi(h_{x_c}, |u^i(x_c)|, V(x_c)),
\end{equation}
where $h_{x_c}\in\mathbb{R}_+$ signifies the size of the corner.  The quantitative estimate in \eqref{eq:ss1} indicates that if a medium scatterer $(\Omega, V)$ possesses a corner, it not only scatters a generic incident wave nontrivially, but also stably, with the lower bound of the scattering energy depending on the size of the corner, the values of the incident wave and the medium parameter at the corner point. 

\smallskip

{\bf Step 3.}~Choose $u^i=v_g$ satisfying \eqref{eq:app1}. By \eqref{eq:app2} in Lemma~\ref{lem:local1}, one has
\begin{equation}\label{eq:app3}
\|u_\infty(\hat x; v_g)\|_{L^2(\mathbb{S}^{n-1})}\leq C \varepsilon. 
\end{equation}

\smallskip

{\bf Step 4.}~Combining \eqref{eq:ss1} and \eqref{eq:app3}, one has
\begin{equation}\label{eq:app4}
\psi(h_{x_c}, |v_g(x_c)|, V(x_c))\leq \|u_\infty(\hat x; v_g)\|_{L^2(\mathbb{S}^{n-1})}\rightarrow 0\quad \mbox{as}\ \varepsilon\rightarrow 0. 
\end{equation}
Finally, noting that as $\varepsilon\rightarrow 0$, $v_g(x_c)\rightarrow v(x_c)$ (in $L^1$), one can solve from \eqref{eq:app4} that \eqref{eq:vanish1} holds. 

It is clear that the proof of Theorem~\ref{thm:local1} critically depends on the Herglotz approximation. Hence, the ``generic condition" in Theorem~\ref{thm:local1} is actually characterized by the following Herglotz approximation of the transmission eigenfunction $v$: 
\begin{equation}\label{eq:app5}
\|{v-v_{g_j}}\|_{L^2(\Omega)} < e^{-j}, \qquad
    \|{g_j}\|_{L^2(\mathbb S^{n-1})} \leq C (\ln j)^\beta, \ j=1, 2, \ldots,
\end{equation}
where $C$ and $\beta$ are two generic positive constants. By Lemma~\ref{lem:Herg}, the transmission eigenfunction $v$ can always be approximated by a sequence of Herglotz waves $v_{g_j}$, $j=1, 2, \ldots$. However, according to the Addendum of \cite{BL1}, $\|g_j\|_{L^2(\mathbb{S}^{n-1})}$ must be divergent. Hence, \eqref{eq:app5} serves as a kind of regularity condition on the transmission eigenfunction in order for the vanishing property in Theorem~\ref{thm:local1} holds. Indeed, by fixing the convergence rate in the first condition in \eqref{eq:app5}, the more relaxed upper bound on the blowup rate in the second condition in \eqref{eq:app5}, the more relaxed regularity requirement on $v$ in order for the vanishing property in Theorem~\ref{thm:local1} holds. In fact, if there is no upper bound requirement in the second condition in \eqref{eq:app5}, then Theorem~\ref{thm:local1} holds for any transmission eigenfunctions $u, v$ belonging to $H^1(\Omega)$. In \cite{DCL}, the regularity requirement in terms of the Herglotz approximation is further relaxed to be:
\begin{equation}\label{eq:app6}
\|v-v_{g_j}\|_{H^1(\Omega)} \leq j^{-1-\Upsilon },\quad  \|g_j\|_{L^2({\mathbb S}^{1})} \leq C j^{\varrho},\ j=1, 2, \ldots,
\end{equation}
where $\Upsilon$ and $\varrho$ are two generic positive constants. 

The argument discussed above in proving Theorem~4.1 in \cite{BL1} is clearly of a global nature since it relies on Lemma~\ref{lem:local1}. In the Addendum in \cite{BL1} and \cite{B}, a ``local" argument was developed, which works to prove the vanishing property of the following partial-data transmission eigenfunctions:
\begin{equation}\label{eq:th1pn1}
\begin{cases}
\big(\Delta+k^2(1+V)\big) u=0\ &\hspace*{-.3cm} \mbox{in}\ \ \Omega,\medskip\\
(\Delta+k^2) v=0\ &\hspace*{-.3cm} \mbox{in}\ \ \Omega,\medskip\\
u|_{\Gamma_h}=v|_{\Gamma_h},\ \ \partial_\nu u|_{\Gamma_h}=\partial_\nu v|_{\Gamma_h}, 
\end{cases}
\end{equation}
where $V$ is given as that in Theorem~\ref{thm:local1} and $\Gamma_h=\partial S_h\cap\partial\Omega$ with $S_h:=B_h(x_c)\cap\Omega$ being an open neighbourhood of the corner inside $\Omega$. The aforementioned local argument is based on the following integral identity:
\begin{equation}\label{eq:app7}
 k^2\int_{S_h}\big[v-(1+V)u \big](x)\cdot u_0(sx)=\int_{\Lambda_h} \partial_\nu(u-v) u_0-(u-v)\partial_\nu u_0,
\end{equation}
where $\lambda_h:=\partial S_h\backslash\partial\Omega$, and $u_0(s x)$ is a harmonic function with $s$ being a free parameter. The strategy is to take advantage of the 
special form of the test function $u_0(sx)$ as well as its free parameter $s$ to delicately characterize the singularities of $u$ and $v$ in the phase space via the integral identity. In achieving this, it is required that the transmission eigenfunctions $u, v$ are H\"older continuous:
\begin{equation}\label{eq:app8}
u, v\in C^\alpha(\overline{S_h}). 
\end{equation}
Under the regularity condition \eqref{eq:app8}, it is established in the Addendum of \cite{BL1} and \cite{B} that $u(x_c)=v(x_c)=0$. By the PDE regularity theory, it is known that one has the following decomposition (cf. \cite{Dauge88,Grisvard,Cos,CN}):
\begin{equation}\label{eq:decomp1}
u=u_{reg}+u_{sing}, \quad v=v_{reg}+v_{sing},\ \ \mbox{in}\ \ S_h,
\end{equation}
where $u_{sing}$ and $v_{sing}$ account for the $H^1$-singularities of $u, v$, whereas $u_{reg}$ and $v_{reg}$ are at least $H^2$-regular. It is noted that an $H^2$-regular function is H\"older continuous by the standard Sobolev embedding. Hence, the regularity requirement \eqref{eq:app8} is fulfilled or not critically depends on the singular parts in the decomposition \eqref{eq:decomp1}. On the other hand, the singular parts indeed belong to $C^\alpha(\overline{S}_h)$ if $u, v$ are sufficiently regular on $\Gamma_h=\partial S_h\cap\partial\Omega$; see \cite{DCL} for the relevant discussion. However, in the transmission conditions on $\Gamma_h$ in \eqref{eq:th1pn1}, the values of $u|_{\Gamma_h}, v|_{\Gamma_h}$ are not a-priori specified. Hence, the vanishing property of $u/v$ can serve as an indicator of the regularity of $u, v$ on $\Gamma_h$. The regularity point discussed above was first explored in \cite{DCL}. As shown in \cite{DCL}, the regularity requirement in \eqref{eq:app8} is a physical condition since when applying the vanishing property to inverse problems or invisibility problems associated with the physical scattering system \eqref{eq:as1}, such a regularity requirement can always be fulfilled. On the other hand, it is numerically shown in \cite{BLLW} that generically, the transmission eigenfunctions possess the vanishing property near a corner. However, there are also exceptional cases that the singularity of the transmission eigenfunctions produces certain ``localizing phenomena". 

There are some further developments of great importance on the local vanishing properties of the transmission eigenfunctions. In \cite{DCL}, the following partial-data transmission eigenvalue problem was considered:
\begin{equation}\label{eq:th1pn2}
\begin{cases}
\big(\Delta+k^2(1+V)\big) u=0\ &\hspace*{-.3cm} \mbox{in}\ \ \Omega,\medskip\\
(\Delta+k^2) v=0\ &\hspace*{-.3cm} \mbox{in}\ \ \Omega,\medskip\\
u|_{\Gamma_h}=v|_{\Gamma_h},\ \ \partial_\nu u|_{\Gamma_h}=(\partial_\nu v+\eta v)|_{\Gamma_h}, 
\end{cases}
\end{equation}
where $\eta\in L^\infty(\Gamma_h)$. The transmission conditions in \eqref{eq:th1pn2} have a strong physical background in modelling conductive medium bodies; see \cite{DCL} and the references cited therein for related background discussions. In \cite{CX}, the following transmission eigenvalue problem was considered: 
\begin{equation}\label{eq:th1pn2}
\begin{cases}
\nabla\cdot(\sigma\nabla u)+k^2(1+V) u=0\ &\hspace*{-.3cm} \mbox{in}\ \ \Omega,\medskip\\
(\Delta+k^2) v=0\ &\hspace*{-.3cm} \mbox{in}\ \ \Omega,\medskip\\
u|_{\partial\Omega}=v|_{\partial\Omega},\ \ \partial_\nu u|_{\Omega}=\partial_\nu v|_{\Omega}, 
\end{cases}
\end{equation}
and the vanishing properties of either $u/v$ or $\nabla v$ were established under certain conditions. In \cite{BXL, LX2} and \cite{BLin}, the vanishing properties of transmission eigenfunctions were studied in certain scenarios for the transmission eigenvalue problems associated with the Maxwell and Lam\'e systems that arise in electromagnetic and elastic wave scattering, respectively. In all of the aforementioned literature, the studies were concerned with the vanishing properties of the transmission eigenfunctions around a corner point on $\partial\Omega$. In a recent article \cite{BL2}, a more insightful geometric viewpoint was proposed for the vanishing properties of the transmission eigenfunctions. It is shown that if the extrinsic curvature of a boundary point on $\partial\Omega$ is sufficiently large, then the transmission eigenfunctions are generically nearly-vanishing near the high-curvature point. The corner case can be regarded an an extreme case where the curvature is infinite. This geometric viewpoint has been consolidated in our recent study in \cite{Amm0}, where it is shown that the high-curvature part of $\partial\Omega$ can be more stably reconstructed from the far-field measurement. The connection between the near-vanishing of the transmission eigenfunctions around a high-curvature boundary point and the super-resolution reconstruction of the local shape of a medium scatterer $(\Omega, V)$ around its high-curvature boundary point can be roughly described as follows. By the study in \cite{Amm0}, it is known that the scattering information of the high-curvature part of $\partial\Omega$ in the far-field data is more significant than the rest part of the $\partial\Omega$, which is the main reason accounting for the super-resolution local reconstruction mentioned above. On the other hand, by Lemma~\ref{lem:local1}, the transmission eigenfunction $v$ can be regarded as a nearly non-scattering incident wave. Hence, it must be nearly vanishing around a high-curvature point since otherwise, it would generate significant scattering due to the interaction with the high-curvature boundary point, which is a contradiction to the nearly non-scattering fact; see also \eqref{eq:ss1} in the corner case. Such an observation also explains that in many qualitative inverse scattering schemes in imaging the shape of a scatterer, namely $\Omega$, the corner or high-curvature places can be better reconstructed.  

The geometric properties of the transmission eigenfunctions are obviously connected to the invisibility in wave scattering; see our discussion in Section~\ref{sect:3.2} and particularly Conjecture 2. In addition to the aforementioned implications and connections, the local geometric structures of transmission eigenfunctions have been used to derive novel unique identifiability and stability results for Schiffer's inverse shape problem \cite{B,BLin,BL2,BL3,BXL,CDL,DCL}. The Schiffer's inverse shape problem is concerned with recovering the shape of a scatterer independent of its physical content by a single far-field pattern, namely $u_\infty(\hat x; u^i)$ generated by a single incident wave $u^i$. This is rather a huge topic in inverse scattering theory with many existing developments and still full of challenging open problems. We refer to \cite{AL,BL4,CX,CDLZ1,CDLZ2,CS,CY,DLW,DLZZ,EY1,EY2,HNS,LL,LLW2,L2,L3,LPRX,LRX,LT,LX1,LYZ1,LYZ2,LZZ1,LZZ2,LZou1,LZou2,LZou3,LZou4,R1,R2,YYL} and the references cited therein for background discussions and related results in different physical scenarios. It is particularly interesting to point out two specific scenarios. First, in \cite{CLL}, the Schiffer inverse shape problem is completely resolved for the inverse problem associated with the fractional Helmhotlz equation. However, the corresponding argument cannot be extended to the non-fractional case. Second, in \cite{BL3,CDL}, it is proved that not only the shape of a scatterer but also its medium content can be uniquely determined by a single far-field pattern, provided the shape geometry and medium parameter satisfy certain a-priori conditions.

\subsection{Global structure of the transmission eigenfunctions}

The geometric structures of transmission eigenfunctions discussed in the previous subsection are of a local nature. In two recent papers \cite{CHLW1,CHLW2}, a certain global geometric structure was unveiled for the first time. Since one of the research papers for those results is still in progress, we only briefly mention them in what follows, which serves more as an announcement.  

Consider a function $w\in L^2(\Omega)$. It is said to be surface-localized if the following condition is fulfilled for a sufficiently small $\epsilon\in\mathbb{R}_+$: 
\begin{equation}\label{eq:localized1}
\frac{\|w\|_{L^2(\mathcal{N}_\epsilon(\partial\Omega))}}{\|w\|_{L^2(\Omega)}}=1+\mathcal{O}(\epsilon), 
\end{equation}
where 
\begin{equation}\label{eq:localized2}
\mathcal{N}_\epsilon(\partial\Omega):=\{x\in\Omega;\ \mathrm{dist}(x, \partial\Omega)<\epsilon\}. 
\end{equation}
By \eqref{eq:localized1} and \eqref{eq:localized2}, it is easily seen that a surface-localized function with its energy localized on the surface $\partial\Omega$. 

Consider the transmission eigenvalue problem \eqref{eq:th1}. Let $0<k_1^2\leq k_2^2\leq \cdots\leq k_j^2\rightarrow +\infty$ denote the real transmission eigenvalues, and $(u_{k_j}, v_{k_j})$ be the corresponding pair of transmission eigenfunctions associated with the eigenvalue $k_j^2$. It is shown in \cite{CHLW1,CHLW2} that there exists a subsequence $(k_{j_n})\subset (k_j)$ with $k_{j_n}^2\rightarrow +\infty$ as $n\rightarrow +\infty$ such that either $(u_{j_n})$ is a sequence of surface-localized eigenstates or $(v_{j_n})$ is a sequence of surface-localized eigenstates. The existence of surface-localized $(u_{j_n})$ or $(v_{j_n})$ depends on $V$. The discovery not only unveils some intriguing spectral phenomenon that was not known before, but also provides a new perspective on wave localization, which is one of the central topics in wave propagation; see \cite{CHLW1} for more related background discussions. Two interesting applications were also proposed by making use of the surface-localized transmission eigenstates in \cite{CHLW1} including producing a super-resolution wave imaging scheme and generating the so-called pseudo plasmon modes with a potential application to the sensing technology. 

\section{Some topics for future study}

The study of the geometric structures of transmission eigenfunctions is still in its early stage with many intriguing and challenging problems open for further investigation. 
The corresponding study is not only mathematically interesting but also physically important. In addition to the conjectures discussed earlier, we propose several more topics from our perspective that are interesting for further investigation. 

\begin{enumerate}

\item Let us first recall the celebrated Courant's nodal domain theorem for the Dirichlet Laplacian eigenfunction, namely $u\in H_0^1(\Omega)$ such that $-\Delta u=\lambda u$ for $\lambda\in\mathbb{R}_+$. The nodal set for an eigenfunction $u$ is defined as the set of points $x$ such that $u(x)=0$. Courant's nodal domain theorem states that the first Dirichlet eigenfunction doest not change sign in $\Omega$ and the $n$-th eigenfunction (counting multiplicity) $u_n$ has at most $n$ nodal domains. It would be interesting to investigate similar nodal domain properties for the transmission eigenfunction. However, due to the vanishing and surface-localizing properties of the transmission eigenfunctions, it would be more appropriate to investigate the local version of the nodal domain property. That is, for the transmission eigenfunction $v$ in \eqref{eq:th1} associated with $k^2\in\mathbb{R}_+$, if we let $B\Subset\Omega$ be a ball, then the number of nodal domains of $v$ that intersect $B$ should be bounded by a quantity that depends on $k, B, \Omega$ and $V$ (here, one might need to impose a certain regularity condition on $V$ in terms of the smoothness and variations).   

\item According to our earlier discussion, the local vanishing property of transmission eigenfunctions has been mainly established that is related to the extrinsic curvature of a boundary point on $\partial\Omega$. It would be interesting to consider the transmission eigenvalue problem \eqref{eq:thg1}, and study the local and global geometric structures of the corresponding transmission eigenfunctions. It can be expected that the geometric structures are critically related to the intrinsic curvatures of the Riemannian metrics. 

\item It would be interesting to investigate the geometric structures of transmission eigenfunctions associated with more general PDE systems (cf. \eqref{eq:tep1} and \eqref{eq:thp1}) as well as their mathematical and physical applications. 

\end{enumerate}

\section*{Acknowledgment}
The work was supported by a startup grant from City University of Hong Kong and Hong Kong RGC General Research Funds (projects 12301218, 12302919 and 12301420).

\end{document}